\documentclass[leqno, abstracton]{scrartcl}

\usepackage{amssymb}
\usepackage{amsmath}
\usepackage{amscd}
\usepackage{amsthm}

\newtheorem{theorem}{Theorem}
\newtheorem{lemma}{Lemma}

\newtheorem{corollary}{Corollary}
\theoremstyle{definition}
\newtheorem{definition}{Definition}

\newcommand{\ord}{\operatorname{ord}}
\newcommand{\Div}{\mathop{Div}}
\newcommand{\di}{\operatorname{div}}
\newcommand{\CH}{\mathop{CH}}

\newcommand{\ann}{\operatorname{ann}}
\newcommand{\vol}{\operatorname{vol}}
\newcommand{\rk}{\operatorname{rk}}
\newcommand{\ts}{\textstyle}

\newcommand{\wDiv}{\widehat{\Div}}
\newcommand{\wdiv}{\widehat{\di}}
\newcommand{\wh}{\widehat}

\newcommand{\End}{\operatorname{End}}
\newcommand{\Spec}{\operatorname{Spec}}

\newcommand{\map}{\rightarrow}
\newcommand{\Map}{\longrightarrow}

\newcommand{\ol}{\overline}

\DeclareMathOperator{\adeg}{\widehat{deg}}

\DeclareMathOperator{\tr}{tr}

\begin{document}

\title{Arakelov theory of noncommutative arithmetic curves}

\author{Thomas Borek\footnote{Departement Mathematik, ETH Z\"urich, R\"amistrasse 101, 8092 Z\"urich, Switzerland; borek@math.ethz.ch. } }

\maketitle

\begin{abstract}
The purpose of this article is to initiate Arakelov theory in a noncommutative setting. More precisely, we are concerned with Arakelov theory of noncommutative arithmetic curves. A noncommutative arithmetic curve is the spectrum of a $\mathbb Z$-order $\mathcal O$ in a finite dimensional semisimple $\mathbb Q$-algebra. Our first main result is an arithmetic Riemann-Roch formula in this setup. We proceed with introducing the Grothendieck group $\widehat{K}_0(\mathcal O)$ of arithmetic vector bundles on a noncommutative arithmetic curve $\Spec\mathcal O$ and show that there is a uniquely determined degree map $\adeg_\mathcal O:\widehat{K}_0(\mathcal O)\map\mathbb R,$ which we then use to define a height function $H_\mathcal O.$ We prove a duality theorem for the height $H_\mathcal O.$
\end{abstract}


\section{Introduction}

In her thesis [3, 4], Liebend\"orfer studied the question whether it is possible to obtain some kind of Siegel's Lemma also in a situation where the coefficients are taken from a (possibly) noncommutative division algebra of finite dimension over $\mathbb Q.$ She establishes a version of Siegel's Lemma when the coefficients lie in a positive definite rational quaternion algebra $D$ and the solution vectors belong to a maximal order in $D.$ In particular, she introduces heights of matrices with coefficients in $D.$

Our work was initially motivated by the question: can we establish some kind of Arakelov theory of noncommutative arithmetic curves which enables us to reformulate the results of Liebend\"orfer? To answer this question, first of all we have to give a definition of noncommutative arithmetic curves. We did this by asking what conditions we have to impose on a possible candidate in order to obtain a nice Arakelov theory for it. It turns out that for our purposes semisimplicity is the right assumption. This has mainly two reasons. Firstly, given a finite dimensional algebra $A$ over a number field $K,$ the trace form $\tr_{A\mid K}:A\times A\map K,$ $(a,b)\mapsto\tr_{A\mid K}(ab)$ is nondegenerate if and only if $A$ is semisimple. The non-degeneracy of the trace form is crucial in order to get a reasonable measure on the associated real algebra $A_\mathbb R=A\otimes_\mathbb Q\mathbb R.$ This so-called canonical measure on $A_\mathbb R$ is essential in Arakelov theory. Secondly, semisimple algebras over number fields enjoy the important property that their orders are one-dimensional in the sense that prime ideals are maximal. This enables us to prove a product formula in this setting. Without some kind of product formula it is probably impossible to establish a concise Arakelov theory. To cut a long story short, a noncommutative arithmetic curve is the spectrum of a $\mathbb Z$-order in a finite dimensional semisimple $\mathbb Q$-algebra.

We now present our main results and methods. Let $\mathcal O$ be a $\mathbb Z$-order in a finite dimensional semisimple $\mathbb Q$-algebra $A.$ To any $\mathcal O$-module $M$ of finite length and any prime ideal $\frak p$ of $\mathcal O$ we associate a number $\ord_\frak p(M),$ which is a natural generalization of the valuation function $v_\frak p$ known for prime ideals in a Dedekind domain. Actually, $\ord_\frak p$ defines a group homomorphism from the unit group of $A$ onto the integers. This enables us to define the (first) Chow group $CH(\mathcal O)$ of $\mathcal O.$ As well as in the commutative case we can define ideal classes of $\mathcal O,$ but unlike the commutative case the set of all ideal classes does not admit a natural group structure.

It is crucial in Arakelov theory not only to deal with the finite places, i.e. the prime ideals, but also to take the infinite places into account. Since every semisimple algebra $A$ which is not a division algebra has nontrivial zero divisors, it does not admit a valuation in the usual sense, and hence there is no well-defined notion of infinite places of $A.$ This is the reason why we have to find a substitute to describe the infinite part. Recall that the Minkowski space $K_\mathbb R$ of an algebraic number field $K$ is isomorphic to $K\otimes_\mathbb Q\mathbb R$ and that the logarithm induces a homomorphism $\log:(K_\mathbb R)^\times\map\bigoplus_{v\mid\infty}\mathbb Rv,$ where the sum is taken over all infinite places $v$ of $K.$ Hence a possible candidate to describe the infinite part is the unit group $(A_\mathbb R)^\times$ of the real algebra $A_\mathbb R=\mathbb R\otimes_\mathbb QA.$ It turns out that this is a good choice and it allows us to define arithmetic divisors and complete ideal classes of $\mathcal O.$

The main results in the second section are Corollary \ref{productformula} and an arithmetic Riemann-Roch formula for noncommutative arithmetic curves. Corollary \ref{productformula} makes a statement about the elements of the unit group $A^\times,$ which generalizes the product formula satisfied by the nonzero elements of a number field. In Section 3 we introduce arithmetic vector bundles on a noncommutative arithmetic curve $\Spec\mathcal O$ and define the arithmetic Grothendieck group $\wh{K}_0(\mathcal O)$ of $\mathcal O.$ We show that there exists a uniquely determined homomorphism $\wh{\deg}_\mathcal O:\wh{K}_0(\mathcal O)\map\mathbb R,$ called the arithmetic degree, which behaves exactly like the corresponding map for hermitian vector bundles on commutative arithmetic curves. We use the arithmetic degree in Section 4 to define the height $H_\mathcal O(V)$ of a free $A$-submodule $V$ of $A^n.$ If $A$ is a division algebra, then our height coincides (up to normalization) with the height introduced by Liebend\"orfer and R\'emond [5]. It is an important property of height functions that they satisfy duality, i.e. that the height of a subspace equals the height of its orthogonal complement. We establish duality for the height $H_\mathcal O$ under the assumption that $\mathcal O$ is a maximal order in $A$ and that the twisted trace form assumes integral values on $\mathcal O\times\mathcal O.$ Liebend\"orfer and R\'emond [5] have proven duality for their height without the assumption that the twisted trace form assumes integral values on $\mathcal O\times\mathcal O.$ Their proof is rather different than the one in the present article.

\section{Definitions and notations}
In this section, we fix the notation which we will use throughout. The unit group of a ring $B$ with 1 is denoted by $B^\times.$ A \emph{prime ideal} of a ring $B$ is a proper non-zero two-sided ideal $\frak{p}$ in $B$ such that $aBb\not\subset\frak p$ for all $a,b\in B\setminus\frak p.$ The set of all prime ideals of $B$ is denoted by $\Spec B$ and called the (prime) spectrum of $B.$ Given an integral domain $R$ with quotient field $K,$ an $R$\emph{-lattice} is a finitely generated torsionfree $R$-module $L.$ We call $L$ a \emph{full} $R$-lattice in a finite dimensional $K$-vector space $V,$ if $L$ is a finitely generated $R$-submodule in $V$ such that $KL=V,$ where $KL$ is the $K$-subspace generated by $L.$ An $R$\emph{-order} in a finite dimensional $K$-algebra $A$ is a subring $\mathcal O$ of $A$ such that $\mathcal O$ is a full $R$-lattice in $A.$ By a left (right) \emph{$\mathcal O$-lattice} we mean
 a left (right) $\mathcal O$-module which is an $R$-lattice. Specifically, a \emph{full left $\mathcal O$-ideal} in $A$ is a full left $\mathcal O$-lattice in $A.$

From now on we let $K$ denote an algebraic number field and $R$ its ring of integers. If $\mathcal O$ is an $R$-order in a finite dimensional semisimple $K$-algebra $A,$ then we call $\Spec\mathcal O$ a \emph{noncommutative arithmetic curve.} By [7, (22.3)], the prime ideals $\frak p$ of $\mathcal{O}$ coincide with the maximal two-sided ideals of $\mathcal{O}.$ Therefore $\mathcal O/\frak p$ is a simple ring, hence it is isomorphic to a matrix algebra $M_{\kappa_\frak p}(S)$ over a (skew)field $S.$ The natural number $\kappa_\frak p$ is uniquely determined by $\frak p$ and is called the \emph{capacity} of the prime ideal $\frak p.$

\section{A Riemann-Roch formula}

Let $\Spec\mathcal O$ be a noncommutative arithmetic curve. Analogously to the commutative case, we define the \emph{divisor group} of $\mathcal O$ to be the free abelian group
$$\Div(\mathcal O)=\bigoplus_{\frak p}\mathbb Z\frak p$$
over the set of all prime ideals $\frak p$ of $\mathcal O.$ The construction of principal divisors is a little bit more complicated as in the standard case. It is based on the following construction. By the Jordan-H\"older Theorem, every left $\mathcal O$-module $M$ of finite length has an $\mathcal O$-decomposition series
$$0=M_0\subset M_1\subset\dots\subset M_l=M,$$
where the composition factors $S_i=M_i/M_{i-1},$ $i=1,\dots,l,$ are simple left $\mathcal O$-modules and where $l=l_\mathcal O(M)$ is the length of $M.$ The set of composition factors $S_i$ is uniquely determined by $M.$ We claim that the annihilator
$$\frak p_i=\ann_\mathcal O(S_i)=\{x\in\mathcal O\mid xS_i=0\}$$
of $S_i$ is a prime ideal of $\mathcal O.$ Obviously it is a proper nonzero two-sided ideal of $\mathcal O,$ and for any two elements $x,y\in\mathcal O\setminus\frak p_i$ we have $xS_i\neq 0$ and $yS_i\neq 0,$ which implies $\mathcal OyS_i=S_i$ because $S_i$ is a simple left $\mathcal O$-module. Hence $x\mathcal O yS_i=xS_i\neq 0,$ therefore $x\mathcal Oy\not\subset\frak p_i,$ which shows that $\frak p_i$ is indeed a prime ideal of $\mathcal O.$

To every left $\mathcal O$-module $M$ of finite length and every prime ideal $\frak p$ of $\mathcal O,$ we may thus associate a natural number $\ord_\frak p(M),$ which is defined to be the number of composition factors $S$ in the Jordan-H\"older decomposition series of $M$ for which $\ann_\mathcal O(S)=\frak p.$ Note that $\ord_\frak p$ generalize the valuation functions associated to prime ideals of Dedekind domains. Similar functions are also introduced by Neukirch [6, I.12] for orders in number fields.

Suppose we have a short exact sequence
$$0\rightarrow M'\rightarrow M\rightarrow M''\rightarrow 0$$
of left $\mathcal O$-modules of finite length. Then the composition factors of $M$ are those of $M'$ together with those of $M'',$ therefore $\ord_\frak p$ behaves additively on short exact sequences.

We apply this fact to the following situation. Recall that $\mathcal O$ is an $R$-order in a finite dimensional semisimple $K$-algebra $A.$ Let $\frak a\subset\mathcal O$ be an ideal of $\mathcal O$ such that $K\frak a=A,$ and let $x\in\mathcal O$ be no zero divisor. Then $\mathcal Ox$ is a full $R$-lattice in $A$ and we have the following exact sequence of left $\mathcal O$-modules of finite length
$$0\rightarrow\mathcal Ox/\frak ax \rightarrow \mathcal O/\frak ax \rightarrow \mathcal O/\mathcal Ox \rightarrow 0.$$
Since $x$ is not a zero divisor, it follows $\mathcal Ox/\frak ax\cong\mathcal O/\frak a,$ thus
\begin{equation}\label{10}
\ord_\frak p(\mathcal O/\frak ax)=\ord_\frak p(\mathcal O/\mathcal Ox)+\ord_\frak p(\mathcal O/\frak a).
\end{equation}
In particular, $\ord_\frak p$ can be extended to a group homomorphism
$$\ord_\frak p: A^\times\rightarrow\mathbb Z,\; u\mapsto\ord_\frak p(u)=\ord_\frak p(\mathcal O/\mathcal O x)-\ord_\frak p(\mathcal O/\mathcal O r),$$
where $r\in R$ and $x\in\mathcal O$ are such that $ru=x.$ To see that this map is well-defined, we let $x,y\in\mathcal O,$ $r,s\in R$ such that $u=\frac{x}{r}=\frac{y}{s}.$ Since $u$ is a unit, it follows that $x$ and $y$ are no zero divisors. The same is true for $r,s\in R^*\subset A^\times.$ Hence both, $\mathcal O/\mathcal Ox$ and $\mathcal O/\mathcal Or$ are $\mathcal O$-module of finite length and applying (\ref{10}) yields
$$\ord_\frak p(\mathcal O/\mathcal Oxs)=\ord_\frak p(\mathcal O/\mathcal Ox)+\ord_\frak p(\mathcal O/\mathcal Os)$$
and
$$\ord_\frak p(\mathcal O/\mathcal Oyr)=\ord_\frak p(\mathcal O/\mathcal Oy)+\ord_\frak p(\mathcal O/\mathcal Or).$$
But $xs=yr,$ thus
$$\ord_\frak p(\mathcal O/\mathcal Ox)+\ord_\frak p(\mathcal O/\mathcal Os)=\ord_\frak p(\mathcal O/\mathcal Oy)+\ord_\frak p(\mathcal O/\mathcal Or),$$
which shows that $\ord_\frak p(u)$ does not depend on the choice of $x$ and $r$ with $ru=x.$

Likewise one sees that $\ord_\frak p:A^\times\map\mathbb Z$ is additive. The homomorphisms $\ord_\frak p$ provide us with a further homomorphism
$$\di:A^{\times}\rightarrow\Div(\mathcal O),\;u\mapsto\di(u)=\left(\ord_{\frak p}(u)\frak p\right)_\frak p.$$
The elements $\di(u)$ are called \emph{principal divisors} and they form a subgroup of $\Div(\mathcal O)$ which we denote by $\mathcal P(\mathcal O).$ The factor group
$$\CH(\mathcal O)=\Div(\mathcal O)/\mathcal P(\mathcal O)$$
is called the (first) \emph{Chow group} of $\mathcal O.$

Next, we partition the set $J(\mathcal O)$ of all full left $\mathcal O$-ideals in $A$ into ideal classes by placing two ideals $\frak a, \frak b$ in the same class if $\frak a\cong\frak b$ as left $\mathcal O$-modules. Each such isomorphism extends to a left $A$-isomorphism $A=K\frak a\cong K\frak b=A,$ hence is given by right multiplication by some $u\in A^{\times}.$ Thus the ideal class containing $\frak a$ consists of all ideals $\{\frak au\mid u\in A^{\times}\}.$ We denote the set of ideal classes of $\mathcal O$ by $Cl(\mathcal O).$ Given any two full left $\mathcal O$-ideals $\frak a,$ $\frak b$ in $A,$ their product $\frak a\frak b$ is also a full left $\mathcal O$-ideal in $A.$ However when $A$ is noncommutative, the ideal class of $\frak a\frak b$ is not necessarily determined by the ideal class of $\frak a$ and $\frak b.$ Specifically for each $u\in A^\times,$ the ideal $\frak au$ is in the
 same class as $\frak a,$ but possibly $\frak au\frak b$ and $\frak a\frak b$ are in different classes. Hence, if $A$ is noncommutative, the multiplication of $\mathcal O$-ideals does not induce a multiplication on $Cl(\mathcal O).$

However there is still a connection between the set $Cl(\mathcal O)$ and the Chow group of $\mathcal O.$ Namely, since $\mathcal O$ is an $R$-lattice in $A,$ for every full left $\mathcal O$-ideal $\frak a$ in $A$ there is a nonzero $r\in R$ such that $\frak ar$ is a left ideal in $\mathcal O.$ As above, (\ref{10}) implies that the map
$$\ord_\frak p:J(\mathcal O)\rightarrow\mathbb Z,\quad\ord_\frak p(\frak a)=\ord_\frak p(\mathcal O/\frak ar)-\ord_\frak p(\mathcal O/\mathcal Or)$$
does not depend on the choice of the element $r\in R$ with $\frak ar\subset\mathcal O$ and thus is well-defined. This yields a mapping
$$\di:J(\mathcal O)\rightarrow\Div(\mathcal O),\;\frak a\mapsto\di(\frak a)=\left(-\ord_\frak p(\frak a)\frak p\right)_\frak p.$$
For all $\frak a\in J(\mathcal O)$ and all $u\in A^{\times},$ we have $\di(\frak a u)=\di(\frak a)-\di(u),$ so the map $\di$ actually lives on ideal classes and we may write
$$\di:Cl(\mathcal O)\rightarrow\CH(\mathcal O).$$

Note that even when $\mathcal O$ is an order in a number field, $Cl(\mathcal O)$ is a group only if $\mathcal O$ is the maximal order. If $\mathcal O$ is not the maximal order then one has to restrict to invertible $\mathcal O$-ideals, cf. [6, I.12].

Until now we have only dealt with the prime ideals of $\mathcal O,$ but it is crucial in Arakelov theory to take also the infinite places into account. Since every semisimple algebra $A$ which is not a division algebra has nontrivial zero divisors, it does not admit a valuation in the usual sense, and hence there is no well-defined notion of infinite places of $A.$ For this reason we have to find a substitute to describe the infinite part. Recall that the Minkowski space $K_\mathbb R$ of an algebraic number field $K$ is isomorphic to the real vector space $K\otimes_\mathbb Q\mathbb R$ and that the logarithm induces a homomorphism $\log:(K_\mathbb R)^\times\map\bigoplus_{v\mid\infty}\mathbb Rv,$ where the sum is taken over all infinite places $v$ of $K.$ Hence a possible candidate to describe the infinite part is the unit group $(A_\mathbb R)^\times$ of the associated real algebra $A_\mathbb R=\mathbb R\otimes_\mathbb QA.$

In this vein we define the \emph{arithmetic divisor group} of $\mathcal O$ to be
$$\wDiv(\mathcal O)=\Div(\mathcal O)\times\left(A_\mathbb R\right)^{\times},$$
where the group operation is defined component-wise which makes $\wDiv(\mathcal O)$ into a non-abelian group. We write $\ol D=\left(D,D_\infty\right)=\big(\sum_\frak pv_\frak p\frak p,D_\infty\big)$ for the elements of $\wDiv(\mathcal O).$ Clearly the homomorphism $\di:A^\times\rightarrow\Div(\mathcal O)$ extends to a homomorphism
$$\wdiv:A^{\times}\rightarrow\wDiv(\mathcal O),\;u\mapsto\wdiv(u)=\left(\di(u),(1\otimes u)\right).$$
The elements $\wdiv(u)$ are called \emph{arithmetic principal divisors} and they form a subgroup of $\wDiv(\mathcal O),$ which we denote by $\widehat{\mathcal P}(\mathcal O).$ The right cosets
$$\widehat{\mathcal P}(\mathcal O)\ol D,\quad\ol D\in\wDiv(\mathcal O)$$
are called \emph{arithmetic divisor classes} of $\mathcal O,$ and the set of all these classes is denoted by $\widehat{CH}(\mathcal O).$ In other words, $\widehat{CH}(\mathcal O)=\widehat P(\mathcal O)\backslash\widehat{Div}(\mathcal O).$ In general, $\widehat{P}(\mathcal O)$ is not a normal subgroup of $\widehat{Div}(\mathcal O),$ so the set $\widehat{CH}(\mathcal O)$ is usually not a group. Nevertheless, in analogy with the commutative case, there is an exact sequence
$$0\Map(A_\mathbb R)^\times/\mathcal O^\times\overset{a}{\Map}\wh{CH}(\mathcal O)\overset{\pi}{\Map}CH(\mathcal O)\Map 0,$$
where $a$ is given by $[x]\mapsto[(0,-x)]$ and $\pi$ maps $[(D,D_\infty)]$ to $[D].$

By a complete $\mathcal O$-ideal, we understand an element of the set
$$\widehat{J}(\mathcal O)=J(\mathcal O)\times(A_\mathbb R)^\times.$$
We write $\ol{\frak a}=\left(\frak a,\frak a_\infty\right)$ for the elements of $\widehat{J}(\mathcal O).$ There is a right-action of the unit group $A^\times$ on the set $\widehat{J}(\mathcal O)$ given by
$$(\ol{\frak a},u)\mapsto\ol{\frak a}u=\left(\frak au,(1\otimes u^{-1})\frak a_\infty\right).$$
We denote the set of orbits by $\widehat{Cl}(\mathcal O)$ and call it the set of \emph{complete ideal classes} of $\mathcal O.$

Again we have a mapping
$$\wdiv:\widehat{J}(\mathcal O)\rightarrow\wDiv(\mathcal O),\;\ol{\frak a}=(\frak a,\frak a_\infty)\mapsto\wdiv(\ol{\frak a})=\left(\di(\frak a),\frak a_\infty\right).$$
Since $\di:A^\times\rightarrow\Div(\mathcal O)$ is a group homomorphism, we have $\di(u^{-1})=-\di(u),$ therefore $\wdiv(\ol{\frak a}u)=\wdiv(u^{-1})\cdot\wdiv(\ol{\frak a}),$ hence we get a well-defined map
$$\wdiv:\widehat{Cl}(\mathcal O)\rightarrow\widehat{CH}(\mathcal O).$$

The \emph{absolute norm} of a complete $\mathcal O$-ideal $\ol{\frak a}=(\frak a,\frak a_\infty)$ is defined to be the real number
$$\frak N(\ol{\frak a})=\left|N_{A_\mathbb R\mid\mathbb R}(\frak a_\infty)\right|\prod_\frak p\frak N(\frak p)^{\ord_\frak p(\frak a)/\kappa_\frak p},$$
where $\frak N(\frak p)=\sharp(\mathcal O/\frak p)$ is the absolute norm and $\kappa_\frak p$ the capacity of the prime ideal $\frak p$ of $\mathcal O,$ and where $N_{A_\mathbb R\mid\mathbb R}$ denotes the norm map from $A_\mathbb R$ to $\mathbb R.$ For a full left $\mathcal O$-ideal $\frak a$ in $A,$ we set $\frak N(\frak a)=\frak N((\frak a,1)).$ A way to compute this number is given in

\begin{theorem}\label{pf}
If $\frak a$ is a full left $\mathcal O$-ideal in $A,$ then
$$\frak N(\frak a)\overset{(i)}{=}\sharp(\mathcal O/\frak ar)\sharp(\mathcal O/\mathcal Or)^{-1}\overset{(ii)}{=}\sharp(\mathcal O/\frak ar)\left|N_{K\mid\mathbb Q}(r)\right|^{-1},$$
where $r\in R$ is any nonzero element such that $\frak ar\subset\mathcal O.$
\end{theorem}

\begin{proof}
Since $\mathcal O$ is a full $R$-lattice in $A,$ there exists a nonzero $r\in R$ such that $\frak ar\subset\mathcal O.$ By definition, $\ord_\frak p(\frak a)=\ord_\frak p(\frak ar)-\ord_\frak p(r)$ which implies
\begin{equation*}
\begin{split}
\frak N(\frak a)&=\prod_\frak p\frak N(\frak p)^{\ord_\frak p(\frak a)/\kappa_\frak p}\\
&=\prod_\frak p\frak N(\frak p)^{(\ord_\frak p(\frak ar)-\ord_\frak p(r))/\kappa_p}\\
&=\prod_\frak p\frak N(\frak p)^{\ord_\frak p(\frak ar)/\kappa_\frak p}\left(\prod_\frak p\frak N(\frak p)^{\ord_\frak p(r)/\kappa_p}\right)^{-1}.
\end{split}
\end{equation*}
Hence, if the formula $\frak N(\frak b)=\sharp(\mathcal O/\frak b)$ is proven for any ideal $\frak b$ in $\mathcal O,$ then equation (i) is established.

So let us assume that $\frak b$ is an ideal in $\mathcal O,$ and let $S_i$ be the $\mathcal O$-composition factors in the Jordan-H\"older decomposition series of $\mathcal O/\frak b.$ As simple left $\mathcal O$-modules they are of the form $S_i\cong\mathcal O/\frak m_i$ with $\frak m_i$ a maximal left ideal of $\mathcal O.$ If we let $\kappa_i$ denote the capacity of the prime ideal $\frak p_i=\ann_\mathcal O(S_i),$ then $\mathcal O/\frak m_i$ is a minimal left ideal of $\mathcal O/\frak p_i,$ and the simple ring $\mathcal O/\frak p_i$ is isomorphic to the ring of $\kappa_i\times\kappa_i$-matrices over the (skew)field $\End_{\mathcal O/\frak p_i}(\mathcal O/\frak m_i).$ Therefore $\mathcal O/\frak p_i\cong(\mathcal O/\frak m_i)^{\kappa_i},$ whence
$$\frak N(\frak b)=\prod_\frak p\frak N(\frak p)^{\ord_\frak p(\frak b)/\kappa_\frak p}=\prod_i\sharp(\mathcal O/\frak m_i)=\sharp(\mathcal O/\frak b).$$

To prove equation (ii), it remains to show  $|N_{K\mid\mathbb Q}(r)|=\sharp(\mathcal O/\mathcal Or).$ Firstly,  $N_{A\mid\mathbb Q}=N_{K\mid\mathbb Q}\circ N_{A\mid K},$ thus $N_{K\mid\mathbb Q}(x)=N_{A\mid\mathbb Q}(x)$ for every $x\in K.$ Secondly, $N_{A\mid\mathbb Q}(x)=\det_\mathbb Q(\rho_x),$ where $\rho_x:A\rightarrow A$ is right multiplication by $x.$ Since $\mathcal O$ is a full $\mathbb Z$-lattice in $A$ and $r\in\mathcal O,$ we have $\det_\mathbb Q(\rho_r)=\det_\mathbb Z(\rho'_r),$ where $\rho'_r:\mathcal O\rightarrow\mathcal O$ is the restriction of $\rho_r$ to $\mathcal O.$ But it is a basic fact in the theory of finitely generated $\mathbb Z$-modules that $\det_\mathbb Z(\rho'_r)=\sharp(\mathcal O/\mathcal Or).$
\end{proof}

As a corollary, we obtain a noncommutative analogue to the well-known product formula $\prod_v |x|_v =1,$ which holds for any nonzero element $x$ in a number field. In our context the product formula reads:

\begin{corollary}\label{productformula}
Every $u\in A^\times$ satisfies
$$\prod_{\frak p}\frak N(\frak p)^{\ord_{\frak p}(u)/\kappa_\frak p}=|N_{A\mid\mathbb Q}(u)|.$$
\end{corollary}

\begin{proof}
For $u\in A^\times,$ we may write $u=r^{-1}x$ with $r\in R$ and $x\in\mathcal O.$ By definition, $\ord_\frak p(u)=\ord_\frak p(\mathcal Ou),$ hence applying Theorem \ref{pf} yields
$$\prod_{\frak p}\frak N(\frak p)^{\ord_{\frak p}(u)/\kappa_\frak p}=\prod_\frak p\frak N(\frak p)^{\ord_\frak p(\mathcal Ou)/\kappa_\frak p}=\sharp(\mathcal O/\mathcal Our)|N_{K\mid\mathbb Q}(r)|^{-1}.$$
On the other hand we have seen in the last paragraph of the proof of Theorem \ref{pf} that $\sharp(\mathcal O/\mathcal Our)=|N_{A\mid\mathbb Q}(ur)|.$ Since the norm is multiplicative and $N_{K\mid\mathbb Q}(r)=N_{A\mid\mathbb Q}(r),$ the corollary follows.
\end{proof}

In the light of the previous results, we see that the absolute norm defines a map
\begin{equation}\label{norm}
\frak N:\widehat{Cl}(\mathcal O)\rightarrow\mathbb R_+^*.
\end{equation}
Indeed, if $\ol{\frak a}$ and $\ol{\frak a}u$ are two representatives of the same complete ideal class, then
\begin{equation*}
\begin{split}
\frak N(\ol{\frak a}u)&=\left|N_{A_\mathbb R\mid\mathbb R}\left((1\otimes u^{-1})\frak a_\infty\right)\right|\frak N(\frak au)\\
&=\left|N_{A_\mathbb R\mid\mathbb R}\left(1\otimes u^{-1}\right)\right|\left|N_{A_\mathbb R\mid\mathbb R}(\frak a_\infty)\right|\left|N_{A\mid\mathbb Q}(u)\right|\frak N(\frak a).
\end{split}
\end{equation*}
Since $N_{A_\mathbb R\mid\mathbb R}(1\otimes u^{-1})=N_{A\mid\mathbb Q}(u^{-1})=N_{A\mid\mathbb Q}(u)^{-1},$ cf. [7, Exercise 1.2], it follows $\frak N(\ol{\frak a}u)=\frak N(\ol{\frak a}),$ which shows that the map in (\ref{norm}) is well-defined.

There is a homomorphism $\deg:\wDiv(\mathcal O)\rightarrow\mathbb R$ defined by
$$\left(\textstyle{\sum}_{\frak p}v_{\frak p}\frak p,D_\infty\right)\mapsto\deg\left(\textstyle{\sum}_{\frak p}v_{\frak p}\frak p,D_\infty\right)=\sum_{\frak p}\textstyle{\frac{v_\frak p}{\kappa_\frak p}}\log\frak N(\frak p)-\log\left|N_{A_\mathbb R\mid\mathbb R}(D_\infty)\right|.$$
The product formula, Corollary \ref{productformula}, ensures that every arithmetic principal divisor $\wdiv(u)\in\widehat{\mathcal P}(\mathcal O)$ satisfies
$$\deg(\wdiv(u))=\sum_{\frak p}\textstyle{\frac{\ord_\frak p(u)}{\kappa_\frak p}}\log\frak N(\frak p)-\log\left|N_{A\mid\mathbb Q}(u)\right|=0.$$
Therefore $\deg$ induces a map
$$\deg:\widehat{\CH}(\mathcal O)\rightarrow\mathbb R.$$

Putting all together, we finally obtain the commutative diagram
\begin{equation}\label{1}
\begin{CD}
\widehat{Cl}(\mathcal O)@>\frak N>>\mathbb{R}_{+}^{*}\\
@V\wdiv VV        @V-\log VV\\
\widehat{\CH}(\mathcal O)@>\deg>>\mathbb{R}.
\end{CD}
\end{equation}
This commutative diagram generalizes the one valid for orders in number fields, cf. [6, page 192].

If $A$ is a finite dimensional semisimple $K$-algebra, then $A_\mathbb R=\mathbb R\otimes_\mathbb Q A$ is a finite dimensional semisimple $\mathbb R$-algebra, cf. [7, (7.18)], hence it follows from [7, (9.26)] that the (reduced) trace map $\tr_{A_\mathbb R\mid\mathbb R}$ from $A_\mathbb R$ onto $\mathbb R$ gives rise to the symmetric nondegenerate $\mathbb R$-bilinear form
$$A_\mathbb R\times A_\mathbb R\rightarrow\mathbb R,\; (x,y)\mapsto\tr(xy).$$
We call this inner product the (reduced) trace form on $A_\mathbb R.$ It provides a Haar measure on the real vector space $A_\mathbb R,$ which we call the canonical measure on $A_\mathbb R.$ We write $\vol(X)$ for the volume of a subset $X\subset A_\mathbb R$ with respect to the canonical measure on $A_\mathbb R.$  For a full $\mathbb Z$-lattice $L$ in $A_\mathbb R,$ we put
$$\vol(L)=\vol(\Phi(L)),$$
where $\Phi(L)$ is a fundamental domain of $L.$

Consider now an $R$-order $\mathcal O$ in $A.$  Every full left $\mathcal O$-ideal $\frak a$ in $A$ is mapped by the embedding $j:A\map A_\mathbb R,$ $a\mapsto 1\otimes a,$ onto a full $\mathbb Z$-lattice $j(\frak a)$ in $A_\mathbb R.$ If $\ol{\frak a}=(\frak a,\frak a_\infty)$ is a complete $\mathcal O$-ideal then $\frak a_\infty\in\left(A_\mathbb R\right)^{\times},$ thus $\frak a_\infty\cdot j(\frak a)$ is a full $\mathbb Z$-lattice in $A_\mathbb R.$ We set
$$\vol(\ol{\frak a})=\vol\left(\frak a_\infty\cdot j(\frak a)\right)$$
and call the real number
$$\chi(\ol{\frak a})=-\log\vol(\ol{\frak a})$$
the \emph{Euler-Minkowski characteristic} of $\ol{\frak a}.$ Now we are ready to state a Riemann-Roch formula for noncommutative arithmetic curves:

\begin{theorem}
Let $\Spec\mathcal O$ be a noncommutative arithmetic curve. Then every complete $\mathcal O$-ideal $\ol{\frak a}$ satisfies the Riemann-Roch formula
$$\chi(\ol{\frak a})=\deg\big(\wdiv(\ol{\frak a})\big)+\chi(\ol{\mathcal O}),$$
where $\ol{\mathcal O}=(\mathcal O,1).$
\end{theorem}

\begin{proof}
We already know that $\deg\circ\wdiv$ is constant on complete ideal classes, and we claim that the same holds for the Euler-Minkowski characteristic. Indeed, if $\ol{\frak a}u$ is another representative of the complete ideal class of $\ol{\frak a}=(\frak a,\frak a_\infty),$ then
$$\vol(\ol{\frak a}u)=\left|N_{A_\mathbb R\mid\mathbb R}\left((1\otimes u^{-1})\frak a_\infty\right)\right|\vol(\frak a u)=\left|N_{A_\mathbb R\mid\mathbb R}(\frak a_\infty)\right|\vol(\frak a)=\vol(\ol{\frak a}).$$
Therefore we can assume that $\frak a\subset\mathcal O.$ But then, $\vol(\frak a)=\vol(\mathcal O)\sharp(\mathcal O/\frak a).$ This together with Theorem \ref{pf} yields
\begin{equation}\label{rr}
\vol(\ol{\frak a})=\left|N_{A_\mathbb R\mid\mathbb R}(\frak a_\infty)\right|\vol(\frak a)=\left|N_{A_\mathbb R\mid\mathbb R}(\frak a_\infty)\right|\vol(\mathcal O)\frak N(\frak a)=\frak N(\ol{\frak a})\vol(\mathcal O).
\end{equation}
On the other hand the commutative diagram (\ref{1}) tells us that $\deg\big(\wdiv(\ol{\frak a})\big)=-\log\frak N(\ol{\frak a}),$ which in combination with (\ref{rr}) establishes our Riemann-Roch formula.
\end{proof}

The Riemann-Roch formula finishes our study of complete $\mathcal O$-ideals. In the next section we will see how complete $\mathcal O$-ideals are embedded in the more general theory of arithmetic vector bundles on noncommutative arithmetic curves.

\section{Arithmetic vector bundles}
Let $*$ be an involution on a finite dimensional semisimple real algebra $B.$ A \emph{$*$-hermitian metric} on a $B$-module $M$ is a $*$-hermitian form $h:M\times M\rightarrow B$ such that $\tr_{B\mid\mathbb R}\circ h$ is positive definite. Here $*$-hermitian means that $h$ is $B$-linear in the first argument and for all $x,y\in M,$ $h(x,y)=h(y,x)^*.$ Let $S$ be a simple algebra of finite dimension over $\mathbb R.$ An involution $*$ on $S$ is called \emph{positive} if the twisted trace form given by $(x,y)\mapsto\tr_{S\mid\mathbb R}(xy^*)$ is positive definite. With the help of Wedderburn's Structure Theorem, it is easy to see that every finite dimensional simple real algebra admits a positive involution. For more details we refer to [2, Section 5.5]. We call an involution on a semisimple real algebra $B$ positive, if its restriction to each simple component of $B$ is a positive involution. We have the following:

\begin{lemma}\label{l1}
Let $B$ be a finite dimensional semisimple real algebra, and let $*$ be a positive involution on $B.$ Then every $*$-hermitian metric $h$ on $B$ is of the form
$$h(x,y)=x\beta\beta^*y^*,\; x,y\in B,$$
for some $\beta\in B^\times.$
\end{lemma}

\begin{proof}
It is obvious that such an $h$ is $*$-hermitian. Since $\beta$ is a unit, it is not a zero divisor, thus $x\beta\neq 0$ whenever $x\neq 0.$ But $h(x,x)=x\beta(x\beta)^*$ and $*$ is positive, therefore $\tr\circ h$ is positive definite.

On the other hand, if $h$ is a $*$-hermitian metric on $B,$ then $h(x,y)=xh(1,1)y^*$ and $h(1,1)^*=h(1,1).$ The positive definiteness of $\tr\circ h$ ensures that $h(1,1)=\beta\beta^*$ for some $\beta\in B^\times.$
\end{proof}

Since every $B$-module is projective it follows directly from the lemma that every finitely generated $B$-module admits a $*$-hermitian metric.

Let us return to the study of noncommutative arithmetic curves $\Spec\mathcal O.$ Recall that $\mathcal O$ is an $R$-order in a finite dimensional semisimple $K$-algebra $A.$ The associated real algebra $A_\mathbb R=\mathbb R\otimes_\mathbb Q A$ is finite dimensional and semisimple as well, cf. [7, (7.18)]. From now on, we fix a positive involution $*$ on $A_\mathbb R$ and we simply write hermitian instead of $*$-hermitian.

\begin{definition}
An \emph{arithmetic vector bundle} on $\Spec\mathcal O$ is a pair $\ol E=(E,h),$ where $E$ is a left $\mathcal O$-lattice such that $A\otimes_\mathcal OE$ is a free $A$-module, and where $h:E_\mathbb R\times E_\mathbb R\map A_\mathbb R$ is a hermitian metric on the left $A_\mathbb R$-module $E_\mathbb R=A_\mathbb R\otimes_\mathcal OE.$

If $A\otimes_\mathcal OE\cong A$ then $\ol E$ is called an \emph{arithmetic line bundle} on $\Spec\mathcal O.$
\end{definition}

In the rest of the section, we generalize some parts of the third chapter of Neukirch's book [6] to our setup. We can mostly mimic the constructions and proofs done there. Two arithmetic vector bundles $\ol E=(E,h)$ and $\ol E'=(E',h')$ on $\Spec\mathcal O$ are called \emph{isomorphic} if there exists an isomorphism $\phi:E\rightarrow E'$ of left $\mathcal O$-modules which induces an isometry $\phi_\mathbb R:E_\mathbb R\rightarrow E'_\mathbb R,$ i.e. $h'\left(\phi_\mathbb R(x),\phi_\mathbb R(y)\right)=h(x,y)$ for all $x,y\in E_\mathbb R.$

By a short exact sequence
$$0\rightarrow\ol E'\rightarrow\ol E\rightarrow\ol E''\rightarrow 0$$
of arithmetic vector bundles we understand a short exact sequence
$$0\Map E'\overset{\alpha}{\Map}E\overset{\beta}{\Map}E''\Map 0$$
of the underlying left $\mathcal O$-modules which splits isometrically, that is, in the sequence
$$0\Map E'_\mathbb R\overset{\alpha_\mathbb R}{\Map}E_\mathbb R\overset{\beta_\mathbb R}{\Map}E''_\mathbb R\Map 0,$$
$E'_\mathbb R$ is mapped isometrically onto $\alpha_\mathbb R(E'_\mathbb R),$ and the orthogonal complement $\left(\alpha_\mathbb R(E'_\mathbb R)\right)^\perp$ is mapped isometrically onto $E''_\mathbb R.$

We let
$$\wh F_0(\mathcal O)=\bigoplus_{\{\ol E\}}\mathbb Z\{\ol E\}$$
be the free abelian group over the isomorphism classes $\{\ol E\}$ of arithmetic vector bundles on $\Spec\mathcal O.$ In this group we consider the subgroup $\wh R_0(\mathcal O)$ generated by all elements $\{\ol E'\}+\{\ol E''\}-\{\ol E\},$ which arise in a short exact sequence
$$0\rightarrow\ol E'\rightarrow\ol E\rightarrow\ol E''\rightarrow 0$$
of arithmetic vector bundles on $\Spec\mathcal O.$ The factor group
$$\wh K_0(\mathcal O)=\wh F_0(\mathcal O)/\wh R_0(\mathcal O)$$
is called the \emph{Grothendieck group} of arithmetic vector bundles on $\Spec\mathcal O.$ Note that the Grothendieck group $\wh K_0(\mathcal O)$ depends on the chosen involution $*$ on $A_\mathbb R,$ but since we have fixed the involution, we omit to indicate this dependence in the notation.

The connection to the Riemann-Roch theory of the last section is given as follows. Every complete $\mathcal O$-ideal $\ol{\frak a}=\left(\frak a,\frak a_\infty\right)$ gives rise to the particular hermitian metric
$$h_{\ol{\frak a}}(x,y)=x\frak a_\infty(\frak a_\infty)^*y^*,\; x,y\in A_\mathbb R,$$
on $\frak a_\mathbb R=A_\mathbb R\otimes_\mathcal O\frak a=A_\mathbb R.$ We thus obtain the arithmetic line bundle $(\frak a,h_{\ol{\frak a}})$ on $\Spec\mathcal O$ for which we use the notation $L(\ol{\frak a}).$

\begin{theorem}\label{groto}
Let $\Spec\mathcal O$ be a noncommutative arithmetic curve. Then there is a well-defined mapping
$$\widehat{Cl}(\mathcal O)\rightarrow\wh K_0(\mathcal O),\; [\ol{\frak a}]\mapsto [L(\ol{\frak a})].$$
Moreover the elements $[L(\ol{\frak a})]$ generate the Grothendieck group $\wh K_0(\mathcal O).$
\end{theorem}

\begin{proof}
The application $[\ol{\frak a}]\mapsto [L(\ol{\frak a})]$ is independent of the choice of the representative $\ol{\frak a}$ of $[\ol{\frak a}]\in\widehat{Cl}(\mathcal O).$ Indeed, if $\ol{\frak b}=\ol{\frak a}u,$ $u\in A^\times,$ is another element in the complete ideal class of $\ol{\frak a},$ then for all $x,y\in A_\mathbb R$ we have
\begin{equation*}
\begin{split}
h_{\ol{\frak b}}\left(x(1\otimes u),y(1\otimes u)\right)&= x(1\otimes u)\left((1\otimes u^{-1})\frak a_\infty\right)\left((1\otimes u^{-1})\frak a_\infty\right)^*\left(y(1\otimes u)\right)^*\\
&=x\frak a_\infty (\frak a_\infty)^*y^*\\
&=h_{\ol{\frak a}}(x,y).
\end{split}
\end{equation*}
Hence the two line bundles $L(\ol{\frak a})$ and $L(\ol{\frak b})$ are isomorphic and in particular $[L(\ol{\frak a})]=[L(\ol{\frak b})].$

It remains to prove that the elements $[L(\ol{\frak a})]$ generate the Grothendieck group $\wh K_0(\mathcal O).$ Let $[\ol E]\in\wh K_0(\mathcal O).$ By definition, $A\otimes_\mathcal OE$ is a free $A$-module of finite rank $n$ say. If $(b_1,\dots,b_n)$ is a basis of $A\otimes_\mathcal OE,$ then for every $x\in E,$ the element $1\otimes x\in A\otimes_\mathcal OE$ is uniquely expressible in the form $1\otimes x=a_1b_1+\dots+a_nb_n$ with $a_i\in A.$ Let $\frak a_i$ be the set of all coefficients $a_i$ which occur as $x$ ranges over all elements of $E.$ This is a finitely generated left $\mathcal O$-submodule of $A.$ Since
$$Ab_1+\dots+Ab_n=A\otimes_\mathcal OE=A(\frak a_1b_1+\dots+\frak a_nb_n)=A\frak a_1b_1+\dots+A\frak a_nb_n,$$
it follows that $A\frak a_i=A$ for every $i,$ so each $\frak a_i$ is a full left $\mathcal O$-ideal in $A.$

We write $F=\frak a_1\oplus\dots\oplus\frak a_{n-1}$ and consider the short exact sequence
$$0\rightarrow F\rightarrow E\rightarrow\frak a_n\rightarrow 0$$
of $\mathcal O$-modules. This sequence becomes an exact sequence of arithmetic vector bundles on $\Spec\mathcal O,$ if we restrict the metric on $E_\mathbb R$ to $F_\mathbb R,$ and if we endow $\frak a_\mathbb R=A_\mathbb R$ with the metric, which is induced by the isomorphism $F_\mathbb R^\perp\cong A_\mathbb R.$ But Lemma \ref{l1} tells us that every hermitian metric on $A_\mathbb R$ is of the form $h(x,y)=x\alpha\alpha^*y^*$ for some $\alpha\in\left(A_\mathbb R\right)^\times.$ Hence $[\ol E]=[\ol F]+\left[L\left((\frak a,\alpha)\right)\right].$ By induction on the rank, we get the desired decomposition of $[\ol E].$
\end{proof}

On the one hand we have the homomorphism $\rk:\wh K_0(\mathcal O)\map\mathbb Z$ defined by $\rk([\ol E])=\rk_A(A\otimes_\mathcal OE);$ on the other hand the inclusion $i:R\hookrightarrow\mathcal O$ induces a group homomorphism
$$i_*:\wh K_0(\mathcal O)\rightarrow \wh K_0(R),\; [(E,h)]\mapsto\left[\left(E,\tr_{A_\mathbb R\mid K_\mathbb R}\circ h\right)\right].$$
Both are used in

\begin{theorem}\label{adeg}
Let $\Spec\mathcal O$ be a noncommutative arithmetic curve. There is a unique homomorphism
$$\adeg_\mathcal O:\wh K_0(\mathcal O)\rightarrow\mathbb R$$
which extends the map $\deg\circ\wdiv:\widehat{Cl}(\mathcal O)\rightarrow\mathbb R,$ i.e., which satisfies $\adeg_\mathcal O([L(\ol{\frak a})])=\deg(\wdiv[\ol{\frak a}])$ for all $[\ol{\frak a}]\in\widehat{Cl}(\mathcal O).$ This homomorphism is given by
\begin{equation}\label{adeg1}
\adeg_\mathcal O([\ol E])=\adeg_K(i_*[\ol E])-\rk([\ol E])\adeg_K(i_*[\ol{\mathcal O}]).
\end{equation}
and is called the \emph{arithmetic degree map.}

Furthermore, if $K=\mathbb Q$ then $\adeg_K(i_*[\ol{\mathcal O}])=\chi(\ol{\mathcal O}).$
\end{theorem}

The theorem states that $\adeg_\mathcal O$ behaves exactly like the arithmetic degree map over number fields. Specifically, equation (\ref{adeg1}) generalizes the Riemann-Roch formula for hermitian vector bundles on commutative arithmetic curves, cf. [6, (III.8.2)]. To prove the theorem, we need a lemma from the theory of commutative arithmetic curves, which certainly is well-known. But as we could not find a proof of it in the literature, we give a proof here.

\begin{lemma}\label{endos}
Let $\ol E=(E,h)$ be a hermitian vector bundle on a commutative arithmetic curve $\Spec R,$ and let $\phi,\psi:E_\mathbb R\map E_\mathbb R$ be two endomorphisms. Then
\begin{equation}\label{endos1}
\adeg\left(E,h\circ(\phi\oplus\psi)\right)=\adeg(\ol E)-\ts{\frac{1}{2}}\log\left|\det_\mathbb R(\phi)\right|-\ts{\frac{1}{2}}\log\left|\det_\mathbb R(\psi)\right|.
\end{equation}
\end{lemma}

\begin{proof}
Following Neukirch [6, III.7], we let $i_*\ol E$ denote the hermitian vector bundle on $\Spec\mathbb Z$ obtained by push-forward the bundle $\ol E.$ The Riemann-Roch formula for hermitian vector bundles on arithmetic curves [6, (III.8.2)] asserts
\begin{equation}\label{endos2}
\adeg(\ol E)=\adeg(i_*\ol E)-\rk(E)\adeg(i_*\ol R).
\end{equation}
We recall
\begin{equation}\label{endos3}
\adeg(i_*\ol E)=-\ts{\frac{1}{2}}\log\left|\det\left[\tr_{K_\mathbb R\mid\mathbb R}\circ h(x_i,x_j)\right]_{1\le i,j\le r}\right|,
\end{equation}
where $x_1,\dots,x_r$ is any $\mathbb Z$-basis of $E.$ On the other hand it is an easy exercise in linear algebra to show
\begin{equation*}
\det\left[\tr_{K_\mathbb R\mid\mathbb R}\circ h\left(\phi(x_i),\psi(x_j)\right)\right]_{1\le i,j\le r}=\det(\phi)\det(\psi)\det\left[\tr_{K_\mathbb R\mid\mathbb R}\circ h(x_i,x_j)\right]_{1\le i,j\le r},
\end{equation*}
which in combination with (\ref{endos3}) yields
$$\adeg\left(i_*(E,h\circ(\phi\oplus\psi))\right)=-\ts{\frac{1}{2}}\log\left|\det(\phi)\right|-\ts{\frac{1}{2}}\log\left|\det(\psi)\right|+\adeg\left(i_*\ol E\right).$$
Putting this into the Riemann-Roch formula (\ref{endos2}) establishes (\ref{endos1}).
\end{proof}

\begin{proof}[Proof of Theorem \ref{adeg}]
By Theorem \ref{groto}, the Grothendieck group $\wh K_0(\mathcal O)$ is generated by the elements $[L(\ol{\frak a})],$ $[\ol{\frak a}]\in\widehat{Cl}(\mathcal O),$ hence a homomorphism $\wh K_0(\mathcal O)\rightarrow\mathbb R$ for which the restriction to $\widehat{Cl}(\mathcal O)$ coincide with the map $\deg\circ\wdiv,$ is uniquely determined.

Both, $\widehat{\deg}_K\circ i_*$ and $\rk$ are homomorphisms from $\wh K_0(\mathcal O)$ to $\mathbb R,$ therefore their sum is a homomorphism as well. It thus remains to show that all $[\ol{\frak a}]\in\widehat{Cl}(\mathcal O)$ satisfy
\begin{equation}\label{ad3}
\widehat{\deg}_K\left(i_*\left[L(\ol{\frak a})\right]\right)-\wh{\deg}_K\left(i_*\mathcal O\right)=\deg\left(\wdiv\left[\ol{\frak a}\right]\right)
\end{equation}
For every nonzero $r\in R,$ $\left[L(\ol{\frak a}r)\right]=\left[L(\ol{\frak a})\right],$ so we may assume $\frak a\subset\mathcal O.$ Then
\begin{equation}\label{ad1}
\begin{split}
\adeg_K\left(i_*\left[L(\ol{\frak a})\right]\right)&=\adeg_K\left(\left[\frak a,\tr_{A_\mathbb R\mid K_\mathbb R}\circ h_{\ol{\frak a}}\right]\right)\\
&=-\log\sharp(\mathcal O/\frak a)+\adeg_K\left(\mathcal O,\tr_{A_\mathbb R\mid K_\mathbb R}\circ h_{\ol{\frak a}}\right)\\
&=-\log\frak N(\frak a)-\log\left|N_{A_\mathbb R\mid\mathbb R}(\frak a_\infty)\right|+\adeg_K\left(i_*\ol{\mathcal O}\right).
\end{split}
\end{equation}
The last equality follows from Theorem \ref{pf} and Lemma \ref{endos} using $N_{A_\mathbb R\mid\mathbb R}(\frak a_\infty)=\det_\mathbb R\left(x\mapsto x\frak a_\infty\right).$ On the other hand the commutative diagram (\ref{1}) tells us
\begin{equation}\label{ad5}
\deg\left(\wdiv\left[\ol{\frak a}\right]\right)=-\log\frak N(\ol{\frak a})=-\log\left|N_{A_\mathbb R\mid\mathbb R}(\frak a_\infty)\frak N(\frak a)\right|.
\end{equation}
Combining (\ref{ad1}) and (\ref{ad5}) yields (\ref{ad3}).

The formula $\wh{\deg}_\mathbb Q\left(i_*[\ol{\mathcal O}]\right)=\chi(\ol{\mathcal O})$ follows from Lemma \ref{endos} using $\left|\det_\mathbb R(*)\right|=1.$
\end{proof}

\section{Heights and duality}
The height is a concept which is strongly related to the arithmetic degree. For number fields this is classical. We briefly resume the basic facts. Let $K$ be a number field of degree $d$ over $\mathbb{Q}$ and let $\rho:K\rightarrow\mathbb R^d$ be the embedding of $K$ in the euclidean space $\mathbb R^d$ given by
$$\rho(x)=(\sigma_1(x),\dots,\sigma_r(x),\operatorname{Re}\tau_1(x),\operatorname{Im}\tau_1(x),\dots,\operatorname{Re}\tau_s(x),\operatorname{Im}\tau_s(x)),$$
where $\sigma_1,\dots,\sigma_r:K\rightarrow\mathbb R$ are the real embeddings and $\tau_1,\ol{\tau}_1,\dots,\tau_s,\ol{\tau}_s:K\rightarrow\mathbb C$ the pairs of complex-conjugate embeddings of $K$ in $\mathbb C.$ Using this notation, Schmidt [8, Chapter 3] defines the height of an $m$-dimensional subspace $V\subset K^n$ to be
\begin{equation}\label{height}
H(V)=\left(2^s|\Delta_K|^{1/2}\right)^{-m}\det\left(\rho^n(V\cap R^n)\right),
\end{equation}
where $\Delta_K$ is the discriminant and $R$ the ring of integers of $K.$

To describe the relation to the arithmetic degree, we let $K_\mathbb R$ denote the Minkowski space of $K$ and $*$ the canonical involution on $K_\mathbb R.$ Recall that the Minkowski space of $K$ can be identified with the real vector space $K\otimes_\mathbb Q\mathbb R.$  We call the hermitian form
$$h_n:K_\mathbb R^{n}\times K_\mathbb R^{n}\rightarrow K_\mathbb R^n,\; ((x_{1},\dots,x_{n}),(y_{1},\dots,y_{n}))\mapsto x_{1}y_{1}^*+\dots+x_{n}y_{n}^*$$
the canonical metric on $K_\mathbb R^n.$ With its help the height of a subspace $V$ of $K^n$ can be expressed as an arithmetic degree:
$$\log H(V)=-\widehat{\deg}_K(V\cap R^{n},h_n).$$

Since the arithmetic degree is also available when working with semisimple $K$-algebras, this observation provides a way to make a definition in a more general setting. Let $\mathcal O$ be an $R$-order in a finite dimensional semisimple $K$-algebra $A.$ Fix a positive involution $*$ on $A_\mathbb R=\mathbb R\otimes_\mathbb QA$ and let $h_n$ be the canonical metric on $A_\mathbb R^n,$ that is
$$h_n(x,y)=x_{1}y_{1}^*+\dots+x_{n}y_{n}^*\quad\text{for all }x,y\in A_\mathbb R^n.$$
Then the \emph{height} $H_\mathcal O(V)$ of a free $A$-submodule $V\subset A^{n}$ is defined by
$$\log H_\mathcal O(V)=-\widehat{\deg}_{\mathcal{O}}(V\cap\mathcal{O}^{n},h_n).$$

Note that the height does not depend on the chosen positive involution $*.$ We now explain why. Let $x\mapsto x'$ and $x\mapsto x^*$ be two positive involutions on the semisimple real algebra $A_\mathbb R$ and let $h'=\tr_{A_\mathbb R\mid\mathbb R}\circ h'_n,$ $h^*=\tr_{A_\mathbb R\mid\mathbb R}\circ h^*_n$ be the associated scalar products on the real vector space $A_\mathbb R.$ By definition of positive involutions, the restrictions of these involutions to each simple component $A_i$ of $A_\mathbb R$ are again (positive) involutions. Thus the restriction of the two involutions $x\mapsto x'$ and $x\mapsto x^*$ to $A_i$ differ by an automorphism of $A_i.$ Hence it follows from the Skolem-Noether Theorem, cf. [7, (7.21)], that there is some $u_i\in A_i^\times$ such that for all $y_i\in A_i,$ $y'_i=u_i^{-1}y_i^*u_i.$ This implies that for all $x,y\in A_\mathbb R,$
$$\tr_{A_\mathbb R\mid\mathbb R}(xy')=\sum_i \tr_{A_i\mid\mathbb R}(x_iy'_i)=\sum_i \tr_{A_i\mid\mathbb R}(x_iu_i^{-1}y_i^*u_i)=\sum_i \tr_{A_i\mid\mathbb R}(x_iy_i^*)=\tr_{A_\mathbb R\mid\mathbb R}(xy^*),$$
whence $h'=h^*.$ Applying Theorem \ref{adeg}, this leads to
\begin{equation}\label{height2}
\begin{split}
\log H_{\mathcal O,'}(V)&=-\adeg_\mathcal O(V\cap\mathcal O^n,h'_n)\\
&=-\adeg_\mathbb Q(V\cap\mathcal O^n,h')+\rk(V)\chi(\mathcal O)\\
&=-\adeg_\mathbb Q(V\cap\mathcal O^n,h^*)+\rk(V)\chi(\mathcal O)\\
&=\log H_{\mathcal O,*}(V),
\end{split}
\end{equation}
which proves that the height $H_\mathcal O$ does indeed not depend on the chosen involution $*.$

Equation (\ref{height2}) also shows that if $\vol$ denotes the Haar measure on the real vector space $A_\mathbb R^n$ induced by the scalar product $\tr_{A_\mathbb R\mid\mathbb R}\circ h_n,$  then the height $H_\mathcal O(V)$ of a free $A$-submodule $V$ of $A^n$ can be expressed in terms of volumes of lattices:
\begin{equation}\label{height3}
H_\mathcal O(V)=\frac{\vol(V\cap\mathcal O^n)}{\vol(\mathcal O)^{\rk(V)}}.
\end{equation}
Hence, in the case where $A$ is a division algebra, our height and the height introduced by Liebend\"orfer and R\'emond [5] coincide (up the normalization by the degree $[A:\mathbb Q]$).

It is an important property of the classical height that it satisfies duality, i.e., that the height of a subspace $V$ of $K^n$ equals the height of its orthogonal complement $V^\perp\subset K^n.$ Liebend\"orfer and R\'emond [5] have established duality for their height $H^{\mathcal O,\ast}$ under the assumption that $\mathcal O$ is a maximal order in a finite dimensional rational division algebra. To formulate our duality theorem, we have to introduce some notation.

As usual we let $\mathcal O$ be a $\mathbb Z$-order in a finite dimensional semisimple $\mathbb Q$-algebra $A$ and $h_n$ the canonical metric on $A_\mathbb R^n$ associated to a fixed positive involution $*$ on the semisimple real algebra $A_\mathbb R=A\otimes_\mathbb Q\mathbb R.$ Moreover, we let $\widetilde{\mathcal O}$ denote the lattice dual to $\mathcal O$ with respect to the twisted trace form $\tau:A_\mathbb R\times A_\mathbb R\map\mathbb R,$ $(x,y)\mapsto\tr_{A_\mathbb R\mid\mathbb R}(xy^*),$ and in analogy with (\ref{height3}), we define the height $H_{\widetilde{\mathcal O}}(V)$ of a free $A$-submodule $V$ of $A^n$ as
\begin{equation}\label{height4}
H_{\widetilde{\mathcal O}}(V)=\frac{\vol(V_\mathbb R\cap\widetilde{\mathcal O}^n)}{\vol(\widetilde{\mathcal O})^{\rk(V)}}.
\end{equation}
Finally, the orthogonal complement $V^\perp$ of a free $A$-submodule $V$ of $A^n$ is computed with respect to the canonical metric $h_n$ restricted to $A^n,$ that is, $V^\perp=\{x\in A^n\mid h_n(x,y)=0\text{ for all }y\in V\}.$

\begin{theorem}\label{dual-theorem}
If the twisted trace form $\tau:A_\mathbb R\times A_\mathbb R\map\mathbb R$ assumes integral values on $\mathcal O\times\mathcal O,$ then every free $A$-submodule $V$ of $A^n$ satisfies
\begin{equation}\label{dt1}
H_\mathcal O(V)=H_{\widetilde{\mathcal O}}(V^\perp).
\end{equation}
If, moreover, $\mathcal O$ is a \emph{maximal} order in $A,$ then
\begin{equation}\label{dt2}
H_\mathcal O(V)=H_{\widetilde{\mathcal O}}(V),
\end{equation}
and in particular,
\begin{equation*}
H_\mathcal O(V)=H_{\mathcal O}(V^\perp).
\end{equation*}
\end{theorem}

\begin{proof}
If $V'_\mathbb R$ denotes the orthogonal complement of $V_\mathbb R$ with respect to the scalar product $h=\tr_{A_\mathbb R\mid\mathbb R}\circ h_n$ and vol the volume on the real vector space $A_\mathbb R^n$ induced by $h,$ then it follows from a result [1, Proposition 1(ii)] of Bertrand that
\begin{equation}\label{A}
\vol(V_\mathbb R\cap\mathcal O^n)=\vol\left(V'_\mathbb R \cap\widetilde{\mathcal O}^n\right)\vol(\mathcal O^n),
\end{equation}
cf. the last equation of page 202 in [1]. We claim
\begin{equation}\label{B}
V'_\mathbb R=V_\mathbb R^\perp.
\end{equation}
Obviously, $V^\perp_\mathbb R$ is a subset of $V'_\mathbb R,$ and since $h$ restricted to $V_\mathbb R\times V_\mathbb R$ is nondegenerate, it follows $V'_\mathbb R \oplus V_\mathbb R=A_\mathbb R^n.$ But the non-degenracy of $h$ implies the non-degeneracy of $h_n$ restricted to $V_\mathbb R\times V_\mathbb R,$ whence $V'_\mathbb R\oplus V_\mathbb R=A_\mathbb R^n$ and thus $V^\perp_\mathbb R=V'_\mathbb R.$ Combining $\vol(\widetilde{\mathcal O})=\vol(\mathcal O)^{-1}$ with (\ref{height3}), (\ref{height4}), (\ref{A}) and (\ref{B}) establishes (\ref{dt1}).

Looking at (\ref{height3}) and (\ref{height4}) and again applying $\vol(\widetilde{\mathcal O})=\vol(\mathcal O)^{-1}$, one sees that (\ref{dt2}) is equivalent to
$$\vol(\mathcal O)^{\rk V}=\vol(V\cap\mathcal O^n) / \vol(V_\mathbb R\cap\widetilde{\mathcal O}^n).$$
But $\vol(V\cap\mathcal O^n)/\vol(V_\mathbb R\cap\widetilde{\mathcal O}^n)=[V_\mathbb R\cap\widetilde{\mathcal O}^n:V\cap\mathcal O^n]$ and $\vol(\mathcal O)=[\widetilde{\mathcal O}:\mathcal O],$ so we are done once
\begin{equation}\label{duality10}
 \left[V_\mathbb R\cap\widetilde{\mathcal O}^n:V\cap\mathcal O^n\right]=\left[\widetilde{\mathcal O}:\mathcal O\right]^{\rk V}
\end{equation}
is established.

We are going to prove (\ref{duality10}) locally. Since $\mathcal O$ is a \emph{maximal} $\mathbb Z$-order, it follows from [7, (18.10)] that for every prime ideal $p$ of $\mathbb Z$ the localization $\mathcal O_p=\mathcal O\otimes_\mathbb Z\mathbb Z_p$ of $\mathcal O$ at $p$ is a left and right principal ideal ring. Therefore, $\widetilde{\mathcal O}_p$ is a free right and $V\cap\mathcal O_p^n$ a free left $\mathcal O_p$-module. This implies
$$V_\mathbb R\cap\widetilde{\mathcal O}_p^n=\widetilde{\mathcal O}_p\left(V_\mathbb R\cap\mathcal O_p^n\right)=\widetilde{\mathcal O}_p\left(V\cap\mathcal O_p^n\right),$$
and thus,
$$\left[V_\mathbb R\cap\widetilde{\mathcal O}_p^n:V\cap\mathcal O_p^n\right]=\left[\widetilde{\mathcal O}_p\mathcal O_p^{\rk V}:\mathcal O_p^{\rk V}\right]=\left[\widetilde{\mathcal O}_p:\mathcal O_p\right]^{\rk V},$$
which establishes (\ref{duality10}) and completes the proof.
\end{proof}

For the sake of completeness, we note that if one works with the Haar measure $\vol'$ on $A_\mathbb R$ induced by the nondegenerate $\mathbb R$-bilinear form $b=\tr_{A_\mathbb R\mid\mathbb R}\circ b_n,$ where $b_n(x,y)=x_1y_1+\dots+x_ny_n,$ $x,y\in A_\mathbb R^n,$ and puts $H'_\mathcal O(V)=\vol'(V\cap\mathcal O^n)/\vol'(\mathcal O)^{\rk V},$ then the height $H'_\mathcal O$ satisfies duality whenever $\mathcal O$ is a maximal order; i.e., the assumption that the twisted trace form assumes integral values on $\mathcal O\times \mathcal O$ becomes needless.

\section*{Acknowledgments}
I would like to thank Gisbert W\"ustholz, J\"urg Kramer, Christine Liebend\"orfer and Ulf K\"uhn for interesting and helpful discussions and comments.

\end{document}